\documentclass[12pt]{amsart} 

\usepackage{amscd} 
\usepackage{amsmath} 
\usepackage{amssymb} 
\usepackage{amsthm}
\usepackage{bigdelim}
\usepackage{color} 
\usepackage{enumerate}
\usepackage{mathrsfs}
\usepackage{multirow}
\usepackage[all]{xy} 
\usepackage[bookmarks=false]{hyperref} 
\newtheorem{thm}{Theorem}[section] 

\newtheorem{prop}[thm]{Proposition}

\newtheorem{lem}[thm]{Lemma}
\theoremstyle{definition} 
\newtheorem{defn}[thm]{Definition}

\newtheorem{step}{Step}
\theoremstyle{remark}

\newtheorem*{ack}{Acknowledgments}

\date{\today}

\title{On the abundance theorem for \\ 
numerically trivial canonical divisors \\
in positive characteristic}
\author{Sho Ejiri}
\address{Department of Mathematics, Graduate School of Science, Osaka Metropolitan University, Osaka City, Osaka 558-8585, Japan}
\email{shoejiri.math@gmail.com}


\begin{document}
\maketitle
\begin{abstract}
In this paper, we prove the abundance theorem for numerically trivial canonical divisors on strongly $F$-regular varieties, assuming that the geometric generic fibers of the Albanese morphisms are strongly $F$-regular. 
\end{abstract}
\setcounter{tocdepth}{1}
%
\markboth{SHO EJIRI}{The abundance theorem for numerically trivial canonical divisors in char p}
\section{Introduction} \label{section:intro}
The abundance conjecture predicts that if a variety is minimal then its canonical divisor is semi-ample. 
This conjecture is one of the most important problems in the minimal model theory. 
In characteristic zero, the conjecture has been verified 
when the variety has only canonical singularities and the canonical divisor is numerically trivial by Kawamata~\cite[Theorem~8.2]{Kaw85}. 
This has been generalized to klt pairs by Nakayama~\cite[V.4.8~Theorem]{Nak04}, 
and to lc pairs by Campana--Koziarz--P\u aun~\cite[Theorem~0.1]{CKP12}, Gongyo~\cite[Theorem~1.2]{Gon13} and Kawamata~\cite[Theorem~1]{Kaw13}. 
 
In this paper we prove the theorem below, which can be viewed as a positive characteristic analog of Nakayama's theorem mentioned above. 
\begin{thm} \label{thm:main}
Let $X$ be a normal projective variety over an algebraically closed field of characteristic $p>0$. 
Let $\Delta$ be an effective $\mathbb Q$-Weil divisor on $X$ such that $(X, \Delta)$ is strongly $F$-regular.  
Assume that $m(K_X +\Delta)$ is Cartier for an integer $m>0$ not divisible by $p$, and that $K_X+\Delta$ is numerically trivial.  
Let $\alpha:X\to A$ be the Albanese morphism of $X$ and let $X_{\overline\eta}$ denote the geometric generic fiber of $\alpha$ over its image. 
If $(X_{\overline\eta},\Delta|_{X_{\overline\eta}})$ is strongly $F$-regular, 
then $K_X+\Delta$ is $\mathbb Q$-linearly trivial. 
\end{thm}
Here, strong $F$-regularity is a class of singularities defined only in positive characteristic, which is closely related to klt singularities in characteristic zero. 
 
We give a brief explanation of the proof of Theorem~\ref{thm:main}. 
For simplicity, we suppose that $\Delta=0$ and $K_X$ is an algebraically trivial Cartier divisor. 
Then $K_X \sim \alpha^*L$ for an algebraically trivial Cartier divisor $L$ on $A$. 
Our goal is to show that $L\sim_{\mathbb Q} 0$. 
By the assumption on the geometric generic fiber of $\alpha$, 
we can apply \cite[Theorem~1.1]{Eji19w} and \cite[Theorems~4.1 and~9.1]{PZ19}, 
so we see that $\alpha$ is a flat surjective morphism whose every fiber is integral and strongly $F$-regular. 
Let $B$ be a sufficiently ample Cartier divisor on $X$. 
Let $r$ be the rank of $\alpha_*\mathcal O_X(B)$. 
We consider the $r$-th fiber product 
$
Y:=X\times_A \cdots \times_A X
$ 
of $X$ over $A$. Let $f:Y\to A$ be the natural projection. 
Applying the argument due to Viehweg~\cite[Proof of Theorem~6.22]{Vie95}, 
we get an effective $f$-ample Cartier divisor $C$ on $Y$ 
such that $\det (f_*\mathcal O_Y(C)) \cong \mathcal O_A$. 
Using \cite[Theorem~6.10~(2)]{Eji19d}, 
we see that $f_*\mathcal O_Y(C)$ is a numerically flat vector bundle. 
Then by Theorem~\ref{thm:homog}, we find unipotent vector bundles $\mathcal U_1, \ldots, \mathcal U_l$ and algebraically trivial line bundles $\mathcal N_1, \ldots, \mathcal N_l$ such that 
$$
f_*\mathcal O_Y(C) \cong \bigoplus_{i=1}^l \mathcal U_i \otimes \mathcal N_i. 
$$
Applying an argument similar to \cite[Proof of Theorem~3.2]{EZ18}, 
we get some $\mu\in\mathbb Z_{>0}$ such that for every 
$\mathcal M \in \mathbb M^{(\mu)}$, where 
$$
\mathbb M^{(\mu)} :=\left\{ \bigotimes_{i=1}^l \mathcal N_i^{n_i} 
	\middle| \textup{$0\le n_i \in \mathbb Z_{}$ with $\sum_{i=1}^l n_i =\mu $}
\right\}, 
$$
we have $\mathcal M \otimes \mathcal O_A(\nu L) \in \mathbb M^{(\mu)}$ 
for infinitely many $\nu \in \mathbb Z_{>0}$. 
Since $\mathbb M^{(\mu)}$ is a finite set, we conclude that $L\sim_{\mathbb Q}0$. 
\begin{ack}
The author is greatly indebted to Professor Hiromu Tanaka for pointing out some typos and for giving him a simple proof of Lemma~\ref{lem:num_triv}.
He wishes to express his thanks to Professors Osamu Fujino, Yoshinori Gongyo, Zsolt Patakfalvi, Shunsuke Takagi and Lei Zhang for valuable comments and helpful advice. 
He would like to thank Professor Kenta Sato for answering his question. 
He also would like to thank the referee for helpful comments. 
He was supported by JSPS KAKENHI Grant Number 18J00171.
\end{ack}

\section{Preliminaries} \label{section:preliminaries}
\subsection{Notation and conventions} \label{subsection:notation}
We work over an algebraically closed field $k$ of characteristic $p>0$. 
A {\it variety} is an integral separated $k$-scheme of finite type. 
Let $\varphi:S\to T$ be a morphism of schemes and let $U$ be a $T$-scheme. 
Let $S_{U}$ and $\varphi_{U}:S_{U}\to U$ denote 
the fiber product $S\times_{T}U$ of $S$ and $U$ 
over $T$ and its second projection, respectively. 
For an $\mathcal O_S$-module $\mathcal G$, 
its pullback to $S_{U}$ is denoted by $\mathcal G_{U}$. 
We use the same notation for a $\mathbb Q$-divisor
if its pullback is well-defined. 
Let $X$ be an $\mathbb F_p$-scheme. 
The Frobenius morphism of $X$ is denoted by $F_X:X\to X$. 
We denote the source of $F_X^e$ by $X^e$. 
Let $f:X\to Y$ be a morphism of $\mathbb F_p$-schemes. 
When we regard $X$ and $Y$ as the sources of $F_X^e$ and $F_Y^e$, respectively, 
$f$ is denoted by $f^{(e)}:X^e\to Y^e$. 
We define $e$-th relative Frobenius morphism of $f$ to be 
$F_{X/Y}^{(e)}:=(F_X^e, f^{(e)}):X^e \to X\times_Y Y^e = X_{Y^e}$.  
\subsection{Numerically flat vector bundles on abelian varieties} \label{subsection:num_flat}
We recall properties of numerically flat vector bundles on abelian varieties. 
\begin{defn} \label{defn:num_flat}
Let $\mathcal E$ be a vector bundle on a projective variety $V$. 
We say that $\mathcal E$ is \textit{numerically flat} if $\mathcal E$ and its dual $\mathcal E^{\vee}$ are nef. 
\end{defn}
One can easily check that the numerical flatness of $\mathcal E$ is equivalent to both $\mathcal E$ and $(\det(\mathcal E))^\vee$ are nef. 
\begin{defn}
Let $\mathcal E$ be a vector bundle on an abelian variety $A$. 
We say that $\mathcal E$ is \textit{homogeneous} if $t_a^*\mathcal E \cong \mathcal E$
for every closed point $a\in A$, where $t_a$ is the translation map of $A$ by $a$. 
\end{defn}
\begin{defn} \label{defn:unipotent}
Let $\mathcal E$ be a vector bundle on an abelian variety $A$. 
We say that $\mathcal E$ is \textit{unipotent} if $\mathcal E$ is an iterated extension of $\mathcal O_A$. 
\end{defn}
\begin{thm} \samepage \label{thm:homog}
Let $A$ be an abelian variety. 
Let $\mathcal E$ be a vector bundle on $A$. 
Then the following are equivalent: 
\begin{itemize}
\item[\rm (1)] $\mathcal E$ is numerically flat; 
\item[\rm (2)] $\mathcal E$ is an iterated extension of algebraically trivial line bundles; 
\item[\rm (3)] $\mathcal E$ is homogeneous vector bundle; 
\item[\rm (4)] $\mathcal E \cong \bigoplus_i \mathcal U_i \otimes \mathcal L_i$, where each $\mathcal L_i$ is an algebraically trivial line bundle and each $\mathcal U_i$ is a unipotent vector bundle. 
\end{itemize}
\end{thm}
\begin{proof}
(1) $\Rightarrow$ (2): This follows from \cite[\S 6]{Lan12}.  
(2) $\Rightarrow$ (4): This is easily proved by induction on the rank. 
(4) $\Rightarrow$ (1): This is obvious. 
(3) $\Rightarrow$ (4): This has been proved by \cite[Theorem~2.3]{Miy73}.
(4) $\Rightarrow$ (3): This has been shown in \cite[Theorem~4.17]{Muk78}. 
\end{proof}
\section{Albanese morphisms of varieties with numerically trivial canonical divisor} \label{section:alb}
In this section, we study the Albanese morphisms of varieties with numerically trivial canonical divisor. 
\begin{lem} \label{lem:num_triv}
Let $V$ be a projective variety over $k$. 
Let $\mathcal L$ be a numerically trivial line bundle on $V$. 
Let $R$ be a finitely generated $\mathbb F_p$-algebra over which the model $V_R$ of $V$ and $\mathcal L_R$ of $\mathcal L$ can be defined. 
Then there exists a dense open subset $S$ of $\mathrm{Spec}\,R$ such that 
$\mathcal L_\mu$ is numerically trivial for every closed point $\mu \in S$. 
\end{lem}
\begin{proof}
Let $S \subseteq \mathrm{Spec}\,R$ be a dense open subset such that 
the morphism $V_R \to \mathrm{Spec}\,R$ is projective over $S$. 
Put $V_S:=\left(V_R\right)_S$. 
Let $\mathcal H_S$ be a line bundle on $V_S$ that is ample over $S$,  
and set $\mathcal H :=\mathcal H_S|_V$.
By \cite[Chapter~I, \S 4, Proposition~3]{Kle66}, 
we have that $\mathcal L$ is numerically trivial if and only if 
$
(\mathcal L \cdot \mathcal H^{\dim V-1})  
= (\mathcal L^2 \cdot \mathcal H^{\dim V-2})  
=0. 
$
Since the intersection number is independent of the choice of fiber, 
the lemma follows. 
\end{proof}
\begin{defn}[\textup{\cite[Definition~3.1]{SS10}}]
Let $X$ be an \textit{affine} normal variety and let $\Delta$ be an effective $\mathbb Q$-Weil divisor on $X$.  
We say that the pair $(X, \Delta)$ is \textit{strongly $F$-regular} if 
for every effective divisor $D$ on $X$, there exists an $e\in\mathbb Z_{>0}$ 
such that the composite
$$
\mathcal O_X 
\xrightarrow{{F_X^e}^\sharp} {F^e_X}_* \mathcal O_X
\xrightarrow{\iota} {F^e_X}_* \mathcal O_X(\lceil (p^e-1)\Delta +D \rceil)
$$
splits. Here, $\iota$ is the natural inclusion. 

Let $X$ be a normal variety and let $\Delta$ be an effective $\mathbb Q$-Weil divisor on $X$. We say that $(X,\Delta)$ is \textit{strongly $F$-regular} if there exists an affine open cover $\{U_i\}$ of $X$ such that $(U_i,\Delta|_{U_i})$ is strongly $F$-regular for each $i$.
When $(X, 0)$ is strongly $F$-regular, 
we simply say that $X$ is \textit{strongly $F$-regular}. 
\end{defn}
Recall if $X$ is strongly $F$-regular, then $X$ is Cohen--Macaulay (\cite[(6.27) Proposition]{HH94}).
\begin{prop} \label{prop:alb}
Let $X$ be a Cohen--Macaulay normal projective variety and let $\Delta$ be an effective $\mathbb Q$-Weil divisor on $X$. 
Assume that $m(K_X+\Delta)$ is a numerically trivial Cartier divisor for an integer $m>0$ not divisible by $p$. 
Let $\alpha:X\to A$ be the Albanese morphism of $X$,
and let $X_{\overline \eta}$ denote the geometric generic fiber of $\alpha$ over its image. 
If $(X_{\overline \eta}, \Delta|_{X_{\overline \eta}})$ is strongly $F$-regular, 
then 
\begin{enumerate}
\item[\rm (1)] $\alpha$ is a flat surjective morphism, 
\item[\rm (2)] $\mathrm{Supp}\,\Delta$ does not contain any component of any closed fiber of $f$, and  
\item[\rm (3)] for every closed fiber $X_y$, the pair $(X_y,\Delta|_{X_y})$ is strongly $F$-regular. 
\end{enumerate}
\end{prop}
%
%
\begin{proof}
By \cite[Theorem~1.1]{Eji19w}, the morphism $\alpha$ is surjective and $X_{\overline\eta}$ is integral. 
By \cite[Theorem~4.1]{PZ19}, $\alpha$ is equi-dimensional, so $\alpha$ is flat, since $X$ is Cohen--Macaulay \cite[p.179, Corollary]{Mat86}. 
Thus (1) holds. 
Furthermore, (2) follows from \cite[Proposition~4.2~(2)]{PZ19}. 
We prove (3). 
Fix a closed point $y\in A$. 
We show that $X_y$ is integral. 
Let $z\in A$ be a general closed point. 
Applying \cite[Th\'eor\`eme~12.2.4]{Gro65}, 
we may assume that $X_z$ is integral. 
Let $R$ be a finitely generated $k$-algebra over which the models $\alpha_R: X_R \to A_R$, $\Delta_R$, $y_R$ and $z_R$ of $\alpha:X\to A$, $\Delta$, $y$ and $z$ can be defined, respectively.
Then, by Lemma~\ref{lem:num_triv}, there is a dense open subset $S\subseteq\mathrm{Spec}\,R$ such that $K_{X_\mu}+\Delta_{\mu}$ is a numerically trivial $\mathbb Q$-Cartier divisor for every $\mu\in S$. 
Since the residue field $\kappa(\mu)$ of $\mu$ is finite, 
we have $\kappa(\mu)\subset \overline{\mathbb F_p}$. 
Therefore, thanks to \cite[Theorem~9.1]{PZ19}, we see that 
$
\left( (X_\mu)_{y_\mu}, (\Delta_\mu)_{y_\mu} \right) 
\cong 
\left( (X_\mu)_{z_\mu}, (\Delta_\mu)_{z_\mu} \right).
$ 
Since $\left( (X_\mu)_{z_\mu}, (\Delta_\mu)_{z_\mu}\right)$ is the reduction of $(X_z, \Delta|_{X_z})$ over $\mu$, 
we may assume that $(X_\mu)_{z_\mu}$ is geometrically integral, using \cite[Th\'eor\`eme~12.2.4]{Gro65}. 
Thus $(X_\mu)_{y_\mu}$ is geometrically integral, 
and so $X_y$ is integral by \cite[Th\'eor\`eme~12.2.4]{Gro65} again. 
Applying \cite[Theorem~B]{PSZ18} 
and an argument similar to the above, 
we can prove that $(X_y,\Delta|_{X_y})$ is strongly $F$-regular. 
\end{proof}
\section{Proof of the main theorem}
In this section, we prove the main theorem in this paper. 
We first show the following three lemmas that are used in the proof of the main theorem. 
\begin{lem} \label{lem:product_SFR}
Let $V$ and $W$ be normal varieties, and let $\Gamma$ and $\Delta$ be effective $\mathbb Q$-Weil divisors on $V$ and $W$, respectively. 
Assume that $(V,\Gamma)$ and $(W,\Delta)$ are strongly $F$-regular.  
Then the pair $(V\times_k W, \mathrm{pr}_1^*\Gamma +\mathrm{pr}_2^*\Delta)$
is also strongly $F$-regular, where $\mathrm{pr}_i$ is the $i$-th projection. 
\end{lem}
Note that since $\mathrm{pr}_i$ is flat, we can define the pullback $\mathrm{pr}_i^* D$ of any $\mathbb Q$-divisor $D$. 
\begin{proof}
We assume that $V$ and $W$ are affine. 
Let $D$ (resp. $E$) be an effective Cartier divisor on $V$ (resp. $W$) such that $(V_D, \Gamma_D)$ 
(resp. $(W_E, \Delta_E)$) 
is log smooth, 
where $V_D := V\setminus \mathrm{Supp}\,D$ and $\Gamma_D := \Gamma|_{V_D}$
(resp. $W_E := W\setminus \mathrm{Supp}\,E$ and $\Delta_E := \Delta|_{W_E}$). 
Here, by log smooth we mean that $V_D$ is smooth and $\Gamma_D$ has simple normal crossing support. 
Set $C := \mathrm{pr}_1^*D + \mathrm{pr}_2^*E$ 
and $\Theta:= \mathrm{pr}_1^*\Gamma +\mathrm{pr}_2^*\Delta$. 
Then 
\begin{align*}
\left( V\times_k W \setminus \mathrm{Supp}\,C,~
\Theta |_{V\times_k W \setminus \mathrm{Supp}\,C} \right)
= 
\left( V_D \times_k W_E,~ (\mathrm{pr}_1|_{V_D})^* \Gamma_D + (\mathrm{pr}_2|_{W_E})^* \Delta_E \right) 
\end{align*}
is log smooth. 
Since
\begin{align*}
\mathcal O_V & \to {F_V^e}_*\mathcal O_V(\lceil(p^e-1)\Gamma +D \rceil ) 
\textup{~and} \\
\mathcal O_W & \to {F_W^e}_*\mathcal O_W(\lceil(p^e-1)\Delta +E \rceil) 
\end{align*}
split for some $e>0$, the morphism 
$$
\mathcal O_{V\times W} \to
{F_{V\times W}^e}_* \mathcal O_{V\times W}(\lceil (p^e-1) \Theta  +C \rceil) 
$$
splits. Hence, the assertion follows from \cite[Theorem~3.9]{SS10}. 
\end{proof}
\begin{lem} \label{lem:surjective}
Let $f:V\to W$ be a surjective morphism between projective varieties. 
Let $A$ be an $f$-ample Cartier divisor on $V$ such that the natural morphism 
$ f^*f_*\mathcal O_V(A) \to \mathcal O_V(A)$ is surjective. 
Let $\mathcal F$ be a coherent sheaf on $V$. 
Then there exists an $m_0\in\mathbb Z_{>0}$ such that the multiplication morphism
$$
f_* (\mathcal F \otimes \mathcal O_V(mA +N) ) \otimes f_* \mathcal O_V(nA)
\to 
f_* (\mathcal F \otimes \mathcal O_V((m+n)A +N) )
$$
is surjective for each $m,n\in\mathbb Z_{>0}$ with $m\ge m_0$ and every $f$-nef Cartier divisor $N$ on $V$.  
\end{lem}
\begin{proof}
By the relative version of Castelnuovo--Mumford regularity \cite[Example~1.8.24]{Laz04I},
it is enough to show that $\mathcal F \otimes \mathcal O_V(mA +N)$ is $0$-regular with respect to $A$ and $f$. 
This follows from the relative Fujita vanishing \cite[Theorem~1.5]{Kee03}. 
\end{proof}
\begin{lem} \label{lem:nef}
Let $\mathcal E$ be a vector bundle on a projective variety $V$. 
Let $H$ be a Cartier divisor on $V$. 
If $\mathcal O_V(H) \otimes {F_V^e}^*\mathcal E$ is nef for infinitely many $e \ge 1$, then $\mathcal E$ is nef. 
\end{lem}
\begin{proof}
Set $P := \mathbb P(\mathcal E) := \mathrm{Proj} ( \bigoplus_{m\ge0} S^m(\mathcal E) )$. 
Let $\pi:P\to V$ be the natural projection. 
Then we have the following commutative diagram:
\begin{align*}
	\xymatrix{ & P^e \ar@/_100pt/[dd]_-{\pi^{(e)}} \ar[d]_-{F_{P/V}^{(e)}} \ar[dr]^-{F_P^e} &  \\
	\mathbb P({F_V^e}^* \mathcal E)	\ar@{=}[r] & P_{V^e} \ar[r]_-{w^{(e)}} \ar[d]_-{\pi_{V^e}} & P \ar[d]^-{\pi} \\ 
	   & V^e \ar[r]_-{F_V^e} & V
}
\end{align*}
Let $T$ be a Cartier divisor on $P$ such that $\mathcal O_P(T)\cong\mathcal O_P(1)$. 
By the assumption, 
$
\mathcal O_{P_{V^e}}(\pi_{V^e}^* H) \otimes \mathcal O_{P_{V^e}}(1) 
$ 
is nef, so 
$$
p^e (p^{-e} {\pi^{(e)}}^*H +T)
= {\pi^{(e)}}^* H +p^e T 
= {F_{P/V}^{(e)}}^*\left(\pi_{V^e}^* H + {w^{(e)}}^* T \right)
$$ 
is nef, and hence $p^{-e} {\pi^*H} +T$ is nef. 
Note that $\mathcal O_{P_{V^e}}(1) \cong {w^{(e)}}^* \mathcal O_P(1)$. 
Since this holds for infinitely many $e$ by the assumption, 
we conclude that $T$ is nef. 
\end{proof}
\begin{defn} \label{defn:L}
Let $\mathcal G$ be a coherent sheaf on a projective variety $V$. 
Then we set 
$$
\mathbb L(\mathcal G)
:=\left\{ \mathcal N \in \mathrm{Pic}^\tau(V) \middle| 
\textup{there is a non-zero morphism $\mathcal G \to \mathcal N$} \right\}. 
$$
\end{defn}
Recall that $\mathrm{Pic}^\tau(V)$ is the set of numerically trivial line bundles on $V$. 
\begin{lem} \label{lem:HN}
Let $\mathcal E$ be a nef vector bundle on a smooth projective variety $V$. 
Then $\mathbb L(\mathcal E)$ is a finite set. 
\end{lem}
\begin{proof}
Let $H$ be an ample divisor on $V$. We set the slope $\mu(\mathcal F)$ 
of a non-zero torsion-free coherent sheaf $\mathcal F$ as 
$$
\mu(\mathcal F) 
:= \frac{c_1(\mathcal F) \cdot H^{n-1}}{\mathrm{rank}(\mathcal F)}.
$$
Let $0=\mathcal E_0 \subset \mathcal E_1 \subset \cdots \subset \mathcal E_n =\mathcal E$ be the Harder--Narasimhan filtration of $\mathcal E$. 
Put $\mathcal E' := \mathcal E/\mathcal E_{n-1}$. 
We first prove that 
$ \mathbb L(\mathcal E) \subseteq \mathbb L(\mathcal E') $. 
Take an $\mathcal N \in \mathbb L(\mathcal E)$. 
We show that $\mathcal N \in \mathbb L(\mathcal E')$. 
By the definition of $\mathbb L(\mathcal E)$, 
there is a non-zero morphism $\mathcal E\to \mathcal N$. 
Let $\iota$ be the largest integer in $\{0,1,\ldots,n-1\}$ such that 
the induced morphism $\mathcal E_\iota \to \mathcal N$ is zero. 
Consider the induced non-zero morphism $\mathcal E_{\iota+1}/\mathcal E_\iota \to \mathcal N$. 
Since $\mathcal E_{\iota+1}/\mathcal E_\iota$ is $\mu$-semistable, 
we have 
$$
0 \le \mu(\mathcal E')
\le \mu(\mathcal E_{\iota+1}/\mathcal E_\iota)
\le \mu(\mathcal N)
= 0. 
$$
Note that the first inequality follows from the nefness of $\mathcal E$. 
Hence we get $\mu(\mathcal E') =\mu(\mathcal E_{\iota+1}/\mathcal E_\iota)=0$, 
which means that $\iota = n-1$, 
so we find a non-zero morphism $\mathcal E' \to \mathcal N$, 
i.e., $\mathcal N \in \mathbb L(\mathcal E')$.   
Next, we prove that $\mathbb L(\mathcal E')$ is a finite set, 
by the induction on the rank. 
When $\mathrm{rank}(\mathcal E') =1$, a non-zero morphism $\mathcal E'\to \mathcal N$ to an $\mathcal N\in\mathrm{Pic}^\tau(V)$ induces 
the non-zero morphism $\mathcal E'^{\vee\vee} \to \mathcal N$ between line bundles, 
which is an isomorphism, as $\mu(\mathcal E')=\mu(\mathcal N) =0$. 
Thus $\mathbb L(\mathcal E') =\{\mathcal E'^{\vee\vee}\}$. 
We next deal with the case when $\mathrm{rank}(\mathcal E') \ge 2$. 
Take $\mathcal N_1, \mathcal N_2 \in \mathbb L(\mathcal E')$ 
with $\mathcal N_1 \not\cong \mathcal N_2$. 
Let $\varphi_i:\mathcal E'\to \mathcal N_i$ 
denote the given non-zero morphisms for $i=1,2$.
We show that the morphism 
$\mathrm{Ker}(\varphi_1) \xrightarrow{\alpha} \mathcal N_2$ 
induced by $\varphi_2$ is non-zero. If this holds, 
then $\mathcal N_2 \in \mathbb L(\mathrm{Ker}(\varphi_1))$, 
which means that 
$
\mathbb L(\mathcal E') 
\subseteq \mathbb L(\mathrm{Ker(\varphi_1)}) \cup \{\mathcal N_1\}, 
$
so the assertion follows from the induction hypothesis. 
Suppose that $\alpha=0$. 
Then we obtain a non-zero morphism $\mathrm{Im}(\varphi_1) \to \mathcal N_2$. 
By the $\mu$-semistability of $\mathcal E'$, we have 
$$
0 \le \mu(\mathcal E') 
\le \mu(\mathrm{Im}(\varphi_1))
\le \mu(\mathcal N_2)
=0, 
$$
so $\mu(\mathrm{Im}(\varphi_1)) = \mu(\mathcal N_2)$, 
and hence $(\mathrm{Im}(\varphi_1))^{\vee\vee} \cong \mathcal N_2$. 
However, by an argument similar to the above, 
we get $(\mathrm{Im}(\varphi_1))^{\vee\vee} \cong \mathcal N_1$, 
so $\mathcal N_1 \cong (\mathrm{Im}(\varphi_1))^{\vee\vee} \cong \mathcal N_2$, 
which contradicts the choice of $\mathcal N_1$ and $\mathcal N_2$. 
Thus we conclude that $\alpha\ne0$.
\end{proof}
\begin{proof}[Proof of Theorem~\ref{thm:main}] \setcounter{step}{0}
By Proposition~\ref{prop:alb}, $\alpha:X\to A$ is a flat surjective morphism with strongly $F$-regular closed fibers. 
Let $s$ be the Cartier index of $K_X+\Delta$. Note that $p\nmid s$. 
By \cite[COROLLARY~6.17]{Kle05}, there is a numerically trivial $\mathbb Q$-Cartier divisor $L$ on $A$ such that $K_X+\Delta \sim_{\mathbb Q} \alpha^*L$. 
Let $t$ be the smallest positive integer such that 
$t(K_X+\Delta)$ and $tL$ are Cartier and $t(K_X+\Delta) \sim_{\mathbb Z} t \alpha^* L$. 
Then $s|t$. 
We put
$$
I := \{i \in \mathbb Z_{>0}| \textup{$0\le i < t$ and $s|i$} \}.
$$
\begin{step} \label{step:B}
Let $B$ be an ample Cartier divisor on $X$. 
Replacing $B$ by $lB$ for $l\gg0$, we may assume that the following conditions hold: 
\setlength{\leftmargini}{20pt}
\begin{enumerate} [(b1)]
\item For every closed point $a \in A$, the natural morphism 
$$
(\alpha_*\mathcal O_X(B))\otimes k(a) 
\to H^0(X_a, \mathcal O_{X_a}(B))
$$
is an isomorphism.
This follows from the proof of \cite[Lemma~3.5]{Eji19p} and Grauert's theorem (see \cite[I\hspace{-1pt}I\hspace{-1pt}I, Corollary~12.9]{Har77}). 
\item For each $m\in\mathbb Z_{>0}$ and each $i \in I$, the sheaf 
$
\alpha_*\mathcal O_X(-i(K_X+\Delta) +mB)
$
is locally free. This follows from the ampleness of $B$ and the flatness of $f$. 
\item For each $m,n\in\mathbb Z_{>0}$ and every $\alpha$-nef Cartier divisor $N$ on $X$, the multiplication morphism 
$$ 
\alpha_* \mathcal O_X(mB +N) \otimes \alpha_* \mathcal O_X(nB) \to \alpha_* \mathcal O_X((m+n)B +N) 
$$ 
is surjective (by Lemma~\ref{lem:surjective}). 
\item For each $e\in\mathbb Z_{>0}$ and every $f$-nef Cartier divisor $N$, 
the morphism 
\begin{align*}
& \alpha^{(e)}_* \mathcal O_X((1-p^e)(K_X+\Delta) +p^e(B+N)) 
\\ & \xrightarrow{\psi_{(X/A,\Delta)}^{(e)}(B_{A^e}+N_{A^e}) }  
{\alpha_{A^e}}_* \mathcal O_{X_{A^e}}(B_{A^e}+N_{A^e}) 
\cong {F_A^e}^* \alpha_*\mathcal O_X(B+N)
\end{align*}
is surjective, where 
$$
\psi_{(X/A,\Delta)}(B_{A^e}+N_{A^e})
:={\alpha_{A^e}}_* \left(\phi_{(X/A,\Delta)}^{(e)}(B_{A^e}+N_{A^e})\right). 
$$
For the construction of the morphism 
$$
\phi_{(X/A,\Delta)}^{(e)}(B_{A^e}+N_{A^e})
:{F_{X/A}^{(e)}}_* \mathcal O_X((1-p^e)(K_X+\Delta) +p^e(B+N))
\to \mathcal O_{X_{A^e}}(B_{A^e}+N_{A^e}), 
$$ 
see \cite[\S 3]{Pat18} or \cite[\S 3]{Eji19d}.
This condition follows from \cite[Lemma~3.7]{Eji19p}. 
Note that, by Proposition~\ref{prop:alb}, $\alpha$ is a flat morphism whose every closed fiber is strongly $F$-regular. 
\end{enumerate}
\end{step}
\begin{step} \label{step:1}
Let $r>0$ be the rank of $\alpha_*\mathcal O_X(B)$. 
We consider the $r$-th fiber product 
$$ 
Y:= X\times_A \cdots \times_A X 
$$
of $X$ over $A$. 
Since $\alpha$ is flat, so is each projection $\mathrm{pr}_i$, 
and hence we can take the pullback $\mathrm{pr}_i^*D$ 
of every $\mathbb Q$-Weil divisor $D$ on $X$. 
Set $\Gamma:=\sum_{i=1}^r \mathrm{pr}_i^*\Delta$. 
Then we get 
\begin{align} \label{relation:1}
t(K_Y +\Gamma)
= t(K_{Y/A} +\Gamma)
\sim \sum_{i=1}^r \mathrm{pr}_i^* t(K_{X/A} +\Delta) 
\sim \sum_{i=1}^r \mathrm{pr}_i^* t\alpha^* L
= rt f^* L. 
\end{align}
Here, $f:Y\to A$ is the natural projection. 
%
Since $(X_a,\Delta|_{X_a})$ is strongly $F$-regular for all closed fibers $X_a$, 
the pair $(Y_a, \Gamma|_{Y_a})$ is also strongly $F$-regular by Lemma~\ref{lem:product_SFR}. 
Put 
$
B':= \sum_{i=1}^r \mathrm{pr}_i^* B. 
$ 
Since $\alpha:X\to A$ is flat, we have 
$$
f_*\mathcal O_Y(B') 
= f_* \left( \bigotimes_{i=1}^r \mathrm{pr}_i^* \mathcal O_X(B) \right)
\cong \bigotimes^r \alpha_* \mathcal O_X(B). 
$$
Since $r$ is the rank of $\alpha_*\mathcal O_X(B)$, we get the morphism
$$
\varphi:
\det \left( \alpha_*\mathcal O_X(B) \right)
\to
\bigotimes_{i=1}^r \alpha_* \mathcal O_X(B) 
$$
that is locally defined as 
$$
x_1 \wedge \cdots \wedge x_r 
\mapsto 
\sum_{\sigma\in \mathfrak S_r} 
\mathrm{sign}(\sigma) \cdot x_{\sigma(1)} \otimes \cdots \otimes x_{\sigma(r)}, 
$$
where $\mathfrak S_r$ is the symmetric group of degree $r$. 
Put $\mathcal H := \det (\alpha_*\mathcal O_X(B))$. 
Let $H$ be a Cartier divisor on $A$ 
with $\mathcal O_A(H) \cong \mathcal H$. 
Then $\varphi$ induces the morphism 
$$
\mathcal O_A
\to 
\mathcal H^{-1} \otimes \left( \bigotimes_{j=1}^r \alpha_* \mathcal O_X(B) \right)
\cong \mathcal H^{-1} \otimes f_*\mathcal O_Y(B') 
\cong f_*\mathcal O_Y(B' -f^*H). 
$$
Therefore, there is an effective $f$-ample Cartier divisor 
$C\sim B' -f^*H$ on $Y$. 
\end{step}
\begin{step} \label{step:fiber}
We show that $\mathrm{Supp}\,C$ does not contain any fiber of $f$. 
Fix a fiber $Y_a$ of $f$. 
Note that $Y_a$ is irreducible, since so is $X_a$. 
We consider the composite $c$ of 
$$ 
c: f^*\mathcal H 
\to \mathcal O_Y(B')
\twoheadrightarrow \mathcal O_{Y_a}(B'), 
$$ 
where the first (resp. second) morphism is the adjoint of 
$\varphi:\mathcal H\to f_*\mathcal O_Y(B')$
(resp. the natural morphism).  
Suppose that $Y_a \subseteq \mathrm{Supp}\,C$. 
Then $c$ is the zero-map. 
Let $V\subseteq A$ be a neighborhood of $a$
such that $(\alpha_*\mathcal O_X(B))|_V$ is free, 
and let $x_1,\ldots, x_r$ be a basis of that. 
Then 
$$
0=c\left( f^*x_1 \wedge \cdots \wedge f^*x_r \right) 
=\sum_{\sigma\in\mathfrak S_r} 
\mathrm{sign}(\sigma) \cdot 
\mathrm{pr}_1^*(x_{\sigma(1)}|_{X_a})
\otimes \cdots \otimes 
\mathrm{pr}_r^*(x_{\sigma(r)}|_{X_a}), 
$$
which means that 
$$
\{ \mathrm{pr}_1^*(x_{\sigma(1)}|_{X_a})
\otimes \cdots \otimes 
\mathrm{pr}_r^*(x_{\sigma(r)}|_{X_a}) \}_{\sigma\in\mathfrak S_r}
\subset
H^0(Y_a, \mathcal O_{Y_a}(B'))
\cong \bigotimes_{i=1}^r H^0(X_a, \mathcal O_{X_a}(B))
$$
is linearly dependent, and so 
$$
\{x_1|_{X_a},\ldots,x_r|_{X_a} \} \subset H^0(X_a, \mathcal O_{X_a}(B))
$$
is also linearly dependent. 
This contradicts that the natural morphism 
$$
( \alpha_*\mathcal O_X(B) )|_V \otimes k(a) 
\to H^0(X_a, \mathcal O_{X_a}(B))
$$
is an isomorphism by (b1). 
\end{step}
\begin{step} \label{step:det}
%
%
We prove that $\det(f_*\mathcal O_Y(C)) \cong \mathcal O_A$. 
This follows from the following calculation: 
\begin{align*}
\det ( f_*\mathcal O_Y(C) )
& \cong \det ( f_*\mathcal O_Y(B'-f^*H) )
\\ & \cong \det \left( \mathcal H^{-1} \otimes f_*\mathcal O_Y(B') \right)
\\ & \cong \det \left( \mathcal H^{-1} \otimes \left( \bigotimes^r \alpha_*\mathcal O_Y(B) \right) \right)
\\ & \cong \mathcal H^{-r^r} \otimes \det \left( \bigotimes^r \alpha_*\mathcal O_Y(B) \right)
= \mathcal H^{-r^r} \otimes \mathcal H^{r^r} \cong \mathcal O_A. 
\end{align*}
\end{step}
\begin{step} \label{step:nef}
We show that 
$
f_*\mathcal O_Y(mC)
$
is a nef vector bundle for each $m\in\mathbb Z_{>0}$ sufficiently large and divisible by using \cite[Theorem~6.10~(2)]{Eji19d}.  
Take $l_0\in\mathbb Z_{>0}$ so that $(Y_a, (\Gamma +l^{-1}C)|_{Y_a})$ is strongly $F$-regular for each $l\ge l_0$ and every closed fiber $Y_a$ of $f$. 
Note that we can find such an $l_0$ thanks to \cite[Theorem~B]{PSZ18}. 
Since $K_Y+\Gamma+l^{-1}C$ is $f$-ample by the construction and $K_{Y/A}=K_Y$, 
we see from \cite[Theorem~6.10~(2)]{Eji19d} that the sheaf 
$$
f_* \mathcal O_Y(lm (K_Y +\Gamma +l^{-1}C))
= f_* \mathcal O_Y(lm (K_Y +\Gamma) +mC)
$$
is weakly positive over $A$ for each $m$ sufficiently large and divisible. 
Fix such an $m$. 
Then 
\begin{align*}
f_* \mathcal O_Y(lm (K_Y +\Gamma) +mC)
\overset{\textup{by (1)}}{\cong} f_* \mathcal O_Y(lmr f^*L +mC)
\cong \mathcal O_A(lmrL) \otimes f_* \mathcal O_Y(mC).
\end{align*}
By (b2), the sheaf $f_* \mathcal O_Y(mC)$ is a vector bundle. 
Since $L$ is numerically trivial, 
we see that $f_*\mathcal O_Y(mC)$ is weakly positive over $A$, 
so it is a nef vector bundle (cf. \cite[Definition~und~Lemma~1.10]{Vie82}).  
\end{step}
\begin{step} \label{step:nef2}
We prove that 
$
f_*\mathcal O_Y(-i(K_Y+\Gamma) +C)
$
is a nef vector bundle for each $i\in I$, 
using Lemma~\ref{lem:nef}. 
For each $a \in \mathbb Z_{>0}$, we have the multiplication morphism 
$$
\mu_a:\bigotimes^a \alpha_*\mathcal O_X(mB) 
\to \alpha_*\mathcal O_X(amB),
$$
which is surjective by (b3). 
Take $e\in\mathbb Z_{>0}$ so that $(1-p^e)(K_Y+\Gamma)$ is Cartier (i.e., $s|(p^e-1)$).
Let $a, b$ be integers such that $p^e = am+b$ and $0\le b < m$. 
Fix $i \in I$. 
We consider the following sequence of morphisms:
\begin{align*}
& f_* \mathcal O_Y((1-p^e-ip^e)(K_Y+\Gamma) +b C) \otimes \left(\bigotimes^a f_*\mathcal O_Y(mC)\right)
\\ \cong &
\mathcal H^{-p^e} \otimes f_* \mathcal O_Y((1-p^e-ip^e)(K_Y+\Gamma) +b B') \otimes \left(\bigotimes^a f_*\mathcal O_Y(mB')\right)
\\ \cong & 
\mathcal H^{-p^e} \otimes \left( \bigotimes^r \left( \alpha_* \mathcal O_X((1-p^e-ip^e)(K_X+\Delta) +bB) \otimes \left( \bigotimes^{a} \alpha_*\mathcal O_X(mB) \right) \right) \right)
\\ \overset{\mu'_1}{\twoheadrightarrow} &
\mathcal H^{-p^e} \otimes \left( \bigotimes^r \left( \alpha_* \mathcal O_X((1-p^e-ip^e)(K_X+\Delta) +bB) \otimes \alpha_*\mathcal O_X(amB) \right) \right)
\\ \overset{\mu'_2}{\twoheadrightarrow} &
\mathcal H^{-p^e} \otimes \left( \bigotimes^{r} \left( \alpha_* \mathcal O_X \left((1-p^e-ip^e)(K_X+\Delta) +p^eB \right) \right) \right)
\\ \cong & 
\mathcal H^{-p^e} \otimes \left( f_* \mathcal O_Y \left((1-p^e-ip^e)(K_Y+\Gamma) +p^eB' \right) \right)
\\ \cong & 
f_* \mathcal O_Y \left((1-p^e-ip^e)(K_Y+\Gamma) +p^eC \right). 
\end{align*}
Here, $\mu_1'$ is induced from $\mu_a$, so $\mu_1'$ is surjective. 
Also, $\mu'_2$ is induced from the multiplication map, 
so it is surjective by (b3), 
where note that $-(K_X+\Delta)$ is nef. 
By (b4), we get the surjection
\begin{align} \label{mor:4}
f_* \mathcal O_Y \left((1-p^e-ip^e)(K_Y+\Gamma) +p^eC \right) 
\twoheadrightarrow 
{F_A^e}^* f_* \mathcal O_Y(-i(K_Y+\Gamma) +C). 
\end{align}
Combining this with the above sequence of morphisms, we obtain the surjection
\begin{align} \label{mor:10}
f_* \mathcal O_Y((1-p^e-ip^e)(K_Y+\Gamma) +b C) 
\otimes \left(\bigotimes^a f_*\mathcal O_Y(mC)\right)
\\ \twoheadrightarrow 
{F_A^e}^* f_* \mathcal O_Y(-i(K_Y+\Gamma) +C). 
\end{align}
For each $i\in I$, let $c_i, d_i$ be integers such that 
$ip^e +p^e-1 = c_it +d_i$ and $0\le d_i < t$. 
We may assume that $e\gg0$ so that $c_i>0$ for each $i\in I$. 
Then 
$$
f_*\mathcal O_Y((1-p^e-ip^e)(K_Y +\Gamma) +bC)
\cong
\mathcal O_A(-c_irtL) \otimes f_*\mathcal O_Y(-d_i(K_Y +\Gamma) +bC). 
$$
Note that $s|d_0$, since $s|(p^e-1)$ and $s|t$. 
Put 
$$
\mathcal G := \bigoplus_{\substack{0\le b < m \\ 0 \le d < t, ~s|d}} 
f_*\mathcal O_Y(-d(K_Y+\Gamma) +bC). 
$$ 
Let $\mathcal K$ be an ample line bundle such that 
$\mathcal G\otimes \mathcal K$ is globally generated.
We have the following sequence of surjections:
\begin{align*}
& \mathcal O_A(-c_irt L) 
\otimes \left(H^0(A, \mathcal G \otimes \mathcal K) \otimes_k \mathcal O_A\right)
\otimes \left(\bigotimes^a f_*\mathcal O_Y(mC) \right)
\\ \twoheadrightarrow & \mathcal O_A(-c_irt L) 
\otimes \mathcal G \otimes \mathcal K
\otimes \left(\bigotimes^a f_*\mathcal O_Y(mC) \right)
\\ \twoheadrightarrow & \mathcal K \otimes \mathcal O_A(-c_irt L) \otimes f_* \mathcal O_Y(-d_i(K_Y+\Gamma) +bC) 
\otimes \left(\bigotimes^a f_*\mathcal O_Y(mC) \right)
\\ \cong & \mathcal K \otimes f_* \mathcal O_Y((1-p^e-ip^e)(K_Y+\Gamma) +bC) 
\otimes \left(\bigotimes^a f_*\mathcal O_Y(mC) \right)
\\ \twoheadrightarrow & \mathcal K \otimes {F_A^e}^*f_*\mathcal O_Y(-i(K_Y+\Gamma) +C)
\end{align*}
Here, the second (resp. fourth) morphism comes from 
the definition of $\mathcal G$ (resp. morphism~(\ref{mor:10})). 
Since $f_*\mathcal O_Y(mC)$ is a nef vector bundle, 
the source of the composite of the above morphisms is nef, 
so the target 
$
\mathcal K \otimes {F_A^e}^* f_*\mathcal O_Y(-i(K_Y+\Gamma)+C)
$ 
is also nef as the composite is surjective.
Then Lemma~\ref{lem:nef} implies that $f_*\mathcal O_Y(-i(K_Y+\Gamma)+C)$ is nef.
\end{step}
\begin{step} \label{step:homog}
We prove the assertion. 
By Step~\ref{step:nef}, the sheaf 
$$
\mathcal F := \bigoplus_{i\in I} f_*\mathcal O_Y(-i(K_Y +\Gamma) +C)
$$
is a nef vector bundle, 
so we see from Lemma~\ref{lem:HN} that $\mathbb L(\mathcal F)$ is a finite set. 
Pick $\mathcal N \in \mathbb L(\mathcal F)$. 
Then there is a non-zero morphism $f_* \mathcal O_Y(-i(K_Y+\Gamma) +C) \to \mathcal N$ for some $i \in I$. 
We consider the following morphisms that are generically surjective:
\begin{align} \label{mor:5}
\mathcal O_A(-c_irtL) \otimes \mathcal F \otimes \left(\bigotimes^{p^e-1}f_*\mathcal O_Y(C) \right)
\overset{\gamma}{\twoheadrightarrow}
{F^e_Y}^*f_*\mathcal O_Y(-i(K_Y+\Gamma) +C) 
\to \mathcal N^{p^e}. 
\end{align}
Here, $\gamma$ is constructed by the same construction as that in Step~\ref{step:nef2}. 
Let $\mathcal E$ denote $f_*\mathcal O_Y(C)$. 
In Step~\ref{step:det} (resp. Step~\ref{step:nef2}), 
we proved that $\det (\mathcal E) =\mathcal O_A$ 
(resp. $\mathcal E$ is nef). 
Hence we see that $\mathcal E$ is numerically flat, 
so Theorem~\ref{thm:homog} tells us that  
there are algebraically trivial line bundles $\mathcal N_{1}, \ldots, \mathcal N_{l}$ and 
unipotent vector bundles $\mathcal U_{1}, \ldots, \mathcal U_{l}$ such that 
\begin{align} \label{cong:Miyanishi}
\mathcal E \cong \bigoplus_{j=1}^{l} \mathcal U_{j} \otimes \mathcal N_{j}. 
\end{align}
Then one can easily check that 
\begin{itemize}
\item[\rm (i)] $\mathbb L(\mathcal E) = \{\mathcal N_1, \ldots, \mathcal N_l\}$, and 
\item[\rm (i\hspace{-1pt}i)] 
there is a filtration
$$
0 =\mathcal E_0 
\subset \mathcal E_1 
\subset \cdots 
\subset \mathcal E_{v-1}
\subset \mathcal E_v 
= \mathcal E
$$
such that $\mathcal E_{j+1}/\mathcal E_j \in \mathbb L(\mathcal E)$ 
for each $j=0,1,\ldots, v-1$. 
\end{itemize}
%
Hence, from (i\hspace{-1pt}i) and morphism~(\ref{mor:5}), 
we obtain the non-zero morphism 
\begin{align} \label{mor:6}
\mathcal O_A(-c_irtL) \otimes \mathcal F \otimes \mathcal P
\to
\mathcal N^{p^e}, 
\end{align}
where $\mathcal P$ is an element of the set $\mathbb M^{(p^e-1)}$ 
that is defined as follows:
for each $\mu\in \mathbb Z_{>0}$, we define $\mathbb M^{(\mu)}$ by 
$$
\mathbb M^{(\mu)} := \left\{ \bigotimes_{j=1}^\lambda \mathcal M_j^{n_j} \middle| \textup{$\mathcal M_j \in \mathbb L(\mathcal F)$ and $0\le n_j \in \mathbb Z_{}$ with $\sum_{j=1}^\lambda n_j =\mu$} \right\}. 
$$
Note that $\mathbb L(\mathcal E) \subseteq \mathbb L(\mathcal F)$. 
From morphism~(\ref{mor:6}), we obtain that 
$$
\mathcal Q := 
\mathcal O_A(c_irtL) \otimes \mathcal N^{p^e} \otimes \mathcal P^{-1} 
\in \mathbb L(\mathcal F),
$$ 
which means that 
$$
\mathcal N^{p^e} \otimes \mathcal O_A(c_irtL) 
\cong \mathcal P \otimes \mathcal Q
\in \mathbb M^{(p^e)}. 
$$
Thus, we see that for each $\mathcal N \in \mathbb L(\mathcal F)$, 
there is $i \in I$ such that 
$$
\mathcal N^{p^e} \otimes \mathcal O_A(c_irtL) \in \mathbb M^{(p^e)}.
$$ 
Set $\lambda:= |\mathbb L(\mathcal F)|$ and $\mu := (p^e-1)\lambda +1$. 
Take $\mathcal M \in \mathbb M^{(\mu)}$. 
Then, we obtain from the pigeonhole principle that 
there is $\mathcal N \in \mathbb L(\mathcal F)$ with
$\mathcal M \otimes \mathcal N^{-p^e} \in \mathbb M^{(\mu-p^e)}$, 
and so
$$
\mathcal M \otimes \mathcal O_A(c_irtL)
\cong
\underbrace{\mathcal M 
\otimes \mathcal N^{-p^e}}_{\in \mathbb M^{(\mu-p^e)}}
\otimes \underbrace{\mathcal N^{p^e}
\otimes \mathcal O_A(c_irtL)}_{\in \mathbb M^{(p^e)}}
\in \mathbb M^{(\mu)} 
$$
for some $c_i$. 
We replace $\mathcal M$ with $\mathcal M \otimes \mathcal O_A(c_irtL)$ 
and repeat the same argument as above, 
then we find an $i_2 \in I$ such that 
$$
\mathcal M \otimes \mathcal O_A((c_i +c_{i_2})rtL) \in \mathbb M^{(\mu)}.
$$
Recall that $c_i>0$ for each $i\in I$ by the choice of $e$.  
Repeating this argument again and again, we get that the set
$$
\mathcal C:=\left\{ c \in \mathbb Z_{>0} 
\middle|  \mathcal M \otimes \mathcal O_A(crtL) \in \mathbb M^{(\mu)} \right\} 
$$
is an infinite set. 
If $L \not\sim_{\mathbb Q} 0$, then 
$
\left\{\mathcal O_Z(crtL) | c\in\mathcal C \right\}
$ 
is an infinite set, so
$$
\left\{\mathcal M \otimes \mathcal O_A(crtL) | c\in\mathcal C\right\} 
$$
is also an infinite set, but this set is contained in $\mathbb M^{(\mu)}$, 
which contradicts the finiteness of $\mathbb M^{(\mu)}$. 
Thus, we conclude that $L \sim_{\mathbb Q} 0$.  
\end{step}
\end{proof}
\bibliographystyle{abbrv}
\bibliography{ref}

\end{document}